\documentclass[a4paper]{article}
\usepackage{amsfonts}
\usepackage{amssymb}
\usepackage{amsmath}
\usepackage{amsthm}
\usepackage{graphicx}
\usepackage{subcaption}
\usepackage{multicol}
\usepackage{siunitx}
\usepackage{url}
\usepackage[utf8]{inputenc}
\usepackage[nottoc]{tocbibind}
\usepackage{tikz}
\usepackage[symbol]{footmisc}
\usepackage{xcolor}
\usepackage{mleftright}
\usepackage{verbatim}
\usepackage[ruled,vlined]{algorithm2e}
\usepackage[noend]{algpseudocode}
\usepackage{lipsum}
\usepackage{mathrsfs}
\usepackage{tikz-cd}
\usepackage{enumitem}

\usepackage{dsfont}
\usepackage{hyperref}
\hypersetup{
	colorlinks=true,
	linkcolor=blue,
	filecolor=magenta,      
	urlcolor=cyan,
}
\usepackage{tabularx}
\usepackage{multirow}
\usepackage{longtable}
\usepackage{emptypage}
\usepackage{cleveref}
\usepackage{enumitem}
\usepackage[toc,page]{appendix}
\usepackage[colorinlistoftodos]{todonotes}
\usepackage{mathtools}

\usepackage{a4wide}

\newtheorem{lemma}{Lemma}[section]

\newtheorem*{theorem*}{Theorem}
\newtheorem{definition}[lemma]{Definition}
\newtheorem{proposition}[lemma]{Proposition}
\newtheorem*{proposition*}{Proposition}

\newtheorem{corollary}[lemma]{Corollary}
\newtheorem{remark}[lemma]{Remark}

\newtheorem{notation}[lemma]{Notation}

\newtheorem{result}{Theorem}

\crefalias{result}{theorem}

\newcommand{\QQ}{\mathbb{Q}}

\newcommand{\ZZ}{\mathbb{Z}}

\newcommand{\NN}{\mathbb{N}}

\newcommand{\inv}{^{-1}}
\newcommand{\comp}{\circ}

\newcommand{\abs}[1]{\left|#1\right|}
\newcommand{\norm}[1]{\left|\left|#1\right|\right|}

\newcommand{\mc}[1]{\mathcal{#1}}

\newcommand{\without}[1]{\backslash\{#1\}}

\newcommand{\conj}{\mathrm{Conj}}
\newcommand{\rf}{\mathrm{Rf}}

\newcommand{\defeq}{\stackrel{\mathclap{\normalfont\mbox{\tiny def}}}{=}}

\title{Independence of the conjugacy problem and conjugacy separability}
\author{Lukas Vandeputte\footnote{The author  kindly acknowledges the support by the group of science Engineering and Technology at KU Leuven Campus Kulak.}}
\begin{document}
	\maketitle
	\sloppy
	\begin{abstract}
		We construct a class of finitely generated groups which have arbitrarily large conjugacy separability function, but in which the conjugacy problem can be solved in polynomial time, demonstrating that the McKinsey algorithm for the conjugacy problem can have complexity which lies arbitrarily far from the optimum.
	\end{abstract}
	\section{Introduction}
	The word problem and the conjugacy problem are two fundamental decision problems in group theory. They ask if it is possible for a given group $G$ with generating set $S$ to determine whether a given word in $S$ is trivial in $G$, or whether two given words are conjugate in $G$ respectively.
	It has been shown for general groups that this is not possible, there exists some explicit finitely presented group $G$ in which it is impossible to determine both of these \cite{Collins1986A}.
	
	On the other hand there are a lot of classes of groups for which these problems are solvable.
	Examples include the finitely generated free groups, Gromov hyperbolic groups, polycyclic groups, braid groups etc. \cite{Epstein2006The,Remeslennikov1971Finite,Garside1969The}.
	Another class of groups for which the word problem and conjugacy problem are solvable, are the recursively-presented groups with quotient enumeration that are respectively
	residually finite or conjugacy separable.
	
	Here a group $G$ is \textbf{residually} finite if for any $g\in G\without{1}$, there exists a finite quotient $\pi:G\rightarrow Q$ such that $\pi(g)\neq 1$. Similarly a group is \textbf{conjugacy separable} if for any pair of non conjugate elements $g_1\nsim g_2\in G$, there exists such a quotient $\pi$ such that $\pi(g_1)\nsim\pi(g_2).$
	
	For these groups, the word problem and conjugacy problem can be solved using a McKinsey type algorithm \cite{McKinsey1943The,Malcev1985}. 	
	To describe the efficiency of these algorithms, one needs to measure the size of the quotients $Q$ above.
	For this, we make use of the \textbf{residual finiteness growth} $\rf_G$ as introduced in \cite{bou2010quantifying} and \textbf{conjugacy separability function} $\conj_G$ as introduced in \cite{Lawton2017Decision}.
	
	These are defined as 
	$$
	\rf_{G,S}= \max \left\{\min\{\abs{Q}\mid \pi:G\rightarrow Q,\pi(g)\neq 1 \}\Big\vert g\in G\without{1},\:\norm{g}_S\leq n\right\}
	$$
	and
	$$
	\conj_{G,S}(n)= \max \left\{\min\{\abs{Q}\mid \pi:G\rightarrow Q,\pi(g_1)\nsim\pi(g_2) \}\Big\vert g_1,g_2\in G, g_1\nsim g_2,\:\norm{g_1}_S,\norm{g_2}_S\leq n\right\}
	$$
	where $\norm{g}_S$ denotes the \textbf{word norm}, i.e. the minimal integer $n$ such that there exist $s_1,s_2,\cdots s_n\in S$ and $\epsilon_1,\epsilon_2,\cdots \epsilon_n\in\{-1,1\}$ satisfying $g=s_1^{\epsilon_1}s_2^{\epsilon_2}\cdots s_n^{\epsilon_n}$.

	In particular, groups with small residual finiteness growth/conjugacy separability function have efficient solutions to the word problem/conjugacy problem. However Kharlampovich, Myasnikov and Sapir demonstrated in \cite{Kharlampovich} that at least in the word problem case, the converse does not hold.

	They constructed a family of groups which simultaneously have a word problem of predetermined complexity, and an arbitrarily high residual finiteness growth.
	
	It is unclear, however, whether the properties of these groups translate well to conjugacy. In particular is it unclear whether these groups have an efficiently solvable conjugacy problem, and whether these groups are conjugacy separable. It thus remains open whether or not an efficient conjugacy problem implies low conjugacy separability.
	
	In this paper, we address this gap by demonstrating the following theorem:
	
	\begin{result}\label{prop:main}
		There exists a family $\mc G$ of recursively-presented, finitely-generated, conjugacy separable groups with quotient enumeration such that for each $G\in \mc G$, the conjugacy problem on $G$ can be solved in polynomial time, but such that for any computable function $f$, there exists some $G\in\mc G$ such that $\conj_G\geq f$.
	\end{result}
	This shows that also for the conjugacy problem, McKinsey type algorithms can be arbitrarily far from optimal.
	
	We start in \cref{sec:thegroup} by describing a bigger class of groups and show in \cref{sec:separability} that these groups are conjugacy separable. Afterwards in \cref{sec:functions} we restrict to a smaller class to guarantee efficient solvability of the conjugacy problem, and we show that these groups have efficient quotient enumeration, and that the residual finiteness function (and thus the conjugacy separability function) for these groups can be arbitrarily large.

	\section{The construction}\label{sec:thegroup}
	Denote the commutator $[g_1,g_2]=g_1g_2g_1\inv g_2\inv$ and denote the conjugate $g^h=h\inv gh$.
	We start by introducing the family of groups we will use throughout the article.
	\begin{definition}
		Let $G_0$ be the group with presentation$$
		\left\langle \{t,a_i,b_i,c_i\mid i\in\ZZ\}\Bigg\vert
		\begin{matrix}
			[[x,y],z]&x,y,z\in\{a_i,b_i\mid i\in \ZZ\}\\
			[a_i,b_j][b_i,a_j]=c_{j-i}&\forall i,j\in \ZZ\\
			ta_it\inv=a_{i+1}&\forall i\in \ZZ\\
			tb_it\inv=b_{i+1}& \forall i\in \ZZ\\
		\end{matrix}
		\right\rangle
		$$
		and let $D_0$ be the group with presentation
		$$
		D=\left\langle \{a_i,b_i,c_i\mid i\in\ZZ\}\Bigg\vert
		\begin{matrix}
			[[x,y],z]&x,y,z\in\{a_i,b_i\mid i\in \ZZ\}\\
			[a_i,b_j][b_i,a_j]=c_{j-i}&\forall i,j\in \ZZ\\
		\end{matrix}
		\right\rangle.
		$$
	\end{definition}
	Notice that $D_0$ is $2$-step nilpotent and that all elements $c_{i}$ are central.
	Let $\varphi_i$ be the automorphism of $D_0$ mapping $a_j$ and $b_j$ to $a_{i+j}$ and $b_{i+j}$ and mapping $c_j$ to $c_j$. One easily checks that this morphism preserves the relations and is thus well defined, furthermore one sees that $\varphi_i\comp\varphi_j=\varphi_{i+j}$. In particular, we can use it to define the semi-direct product $D_0\rtimes_{\varphi}\ZZ$. This group is easily seen to be isomorphic to $G_0$. In what follows, we thus identify the two with one another.
	Notice that $G_0$ is generated by $t,a_0$ and $b_0$ and is thus finitely generated.
	\begin{notation}
		An element of $G_0$ will typically be denoted with either a variable $g$, or with a pair $(h,n)$ which should be interpreted as an element of $D_0\rtimes_\varphi\ZZ$.
	\end{notation}
	
	As the elements $c_i$ are central in $D_0$, and as $\varphi$ leaves them invariant, we have that these elements are central in $G_0$. We define a class of groups, by taking central quotients of $G_0$.
	\begin{definition}
		Let $d:\NN\rightarrow\NN$, then we define the group $G_d$ as$$
		\frac{G_0}{N_d}
		$$
		where $N_d$ is the central subgroup generated by the elements $\{c_{2^{i}}^{d(i)}\mid i\in \NN\}$.
	\end{definition}
	If $d$ is the constant $0$ function then we obtain $G_d=G_0$, motivating our notation.
	Notice that if $d$ is computable, then both $G_0$ and $G_d$ are recursively presented. We also introduce a class of quotients of $G_0$ and $G_d$.
	\begin{definition}
		For $I$ a strictly positive natural number, we define $G_{0,I}$ and $G_{d,I}$ similar to $G_0$ and $G_d$, where the indices $i$ take values in $\frac{\ZZ}{I\ZZ}$ instead of $\ZZ$ and where we impose the additional relation $t^{I}=1$.
		For $G_{0,I}$ we thus obtain the presentation:
		$$
		\left\langle \{t,a_i,b_i,c_i\mid i\in\frac{\ZZ}{I\ZZ}\}\Bigg\vert
		\begin{matrix}
			[[x,y],z]&x,y,z\in\{a_i,b_i\mid i\in \frac{\ZZ}{I\ZZ}\}\\
			[a_i,b_j][b_i,a_j]=c_{j-i}&\forall i,j\in \frac{\ZZ}{I\ZZ}\\
			ta_it\inv=a_{i+1}&\forall i\in \frac{\ZZ}{I\ZZ}\\
			tb_it\inv=b_{i+1}& \forall i\in \frac{\ZZ}{I\ZZ}\\
			t^{I}
		\end{matrix}
		\right\rangle
		$$
		Denote $\pi_I$ the natural projection from $G_{d}$ to $G_{d,I}$ and denote $K_I$ its kernel.
	\end{definition}
	We denote the subgroup $\langle{\{a_i,b_i\mid i\in\frac{\ZZ}{I\ZZ}\}}\rangle=\pi_I(D_0)<G_{d,I}$ by $D_{d,I}$, or just with $D$ if no confusion is possible.
	The groups $G_{d,I}$ are finitely generated virtually nilpotent. Indeed, $\frac{G_{d,I}}{D_{d,I}}$ is finite cyclic and $D_{d,I}$ is nilpotent by definition.
		
	We denote with $D'_{d,I}$ the derived subgroup of $D_{d,I}$ and with $C$ the (central) subgroup of $G_{d,I}$ generated by the elements $c_i$. It is clear that $C$ is contained in $D'$. However, we demonstrate that $C$ does not contain any commutators.
	\begin{lemma}\label{prop:centAndCommutatorsDisjoint}
		Let $g_1,g_2\in G_{d,I}$, then $[g_1,g_2]\in C$ if and only if $[g_1,g_2]=1$.
	\end{lemma}
	\begin{proof}
		It suffices to show the above statement for $d=0$. Notice that $[g_1g_2,g_2]=[g_1,g_2]=[g_1,g_2g_1]$ and thus, using Euclid's algorithm, we can reduce to the case where $g_1=(h_1,n_t)$ and $g_2=(h_2,0)$. 
		First assume $n_t=0$.
		Then we have to show that the same result holds in $D_{0,I}$.
		In this case it suffices to show that in $N_2(a_0,a_1,b_0,b_1)$, the free $2-$step nilpotent group on $4$ generators, if a commutator is of the form $[a_0,b_1]^{d_1}[b_0,a_1]^{d_2}$, then one of the $d_i$ must vanish. This can be checked easily.
				
		Now assume $n_t\neq 0$.
		the commutator $[g_1,g_2]$ is then given by $(h_1 \varphi_{n_t}(h_2)h_1\inv h_2\inv,0)$.
		If $[g_1,g_2]\in C$, then we certainly need that it lies in $D'$ or thus do we have that $\varphi_{n_t}(h_2)h_2\inv \in D'$.
		We can thus rewrite $[g_1,g_2]=(h_1,0) (1,n_t)(h_2,0)(1,-n_t) (h_1\inv,0) (h_2\inv,0)$ as$$
		([h_1,h_2],0)(\varphi_{n_t}(h_2)h_2\inv,0).$$
		
		Suppose that $I=0$.
		In this case, $\varphi_{n_t}(h_2)h_2\inv \in D'_{0,0}$ implies that $h_2\in D'_{0,0}$ and thus $[h_1,h_2]=1$.

		The abelian group $D'$ decomposes as a direct sum $C+K$ where $K$ is generated by elements of the form $[a_i,a_j],[b_i,b_j]$ and $[a_i,b_j][a_j,b_i]$. Notice that $\varphi_{n_t}$ preserves these summands. Furthermore as $\varphi_{n_t}$ leaves elements of $C$ invariant, it follows that if $\varphi_{n_t}(h_2)h_2\inv\in C$, then $\varphi_{n_t}(h_2)h_2\inv=1$.

		Suppose on the other hand that $I\neq 0$ and let $I_0$ be minimal such that $I\mid I_0n_t$.
		As $C$ is torsion free, it suffices to show that $[g_1,g_2]^{I_0}$ vanishes. As $[g_1,g_2]$ is central, this is equivalent to the vanishing of the product \begin{align*}
			&\prod_{i=1}^{I_0} t^{i\cdot n_t}[g_1,g_2]t^{-i\cdot n_t}\\
			=(&\prod_{i=1}^{I_0} \varphi_{i\cdot n_t}[h_1,h_2] \prod_{i=1}^{I_0} \varphi_{(i+1)n_t}(h_2)\varphi_{i\cdot n_t}(h_2)\inv  ,0)			
		\end{align*}
		As $\varphi_{I_0n_t}=\varphi_0$, the second of these products vanishes.
		Now using that $\varphi_{n_t}(h_2)D'=h_2D'$, we can further rewrite this as$$
		([\prod_{i=1}^{I_0}\varphi_{i\cdot n_t}(h_2),h_2],0).
		$$
		As $[g_1,g_2]\in C$, so must be the above element, and thus by the $n_t=0$ case we know that it must vanish.
		
	\end{proof}
	
	\section{Conjugacy separability}\label{sec:separability}
	In this section we demonstrate the following proposition:
	\begin{proposition}\label{prop:conjsep}
		$G_d$ is conjugacy separable for any choice of $d$.
	\end{proposition}
	In what follows, $x,y$ are standins for either $a$ of $b$, e.g. $x_i$ is either $a_i$ or $b_i$.
	Before we give the proof, we first give $2$ lemmas:
	
	\begin{lemma}\label{prop:XPhiXVanishes}
		Let $I,I_0,n_t>0$ be integers such that $I_0$ and $I$ are multiples of $n_t$ and such that $I\geq8I_0$. Let $p\in D_{d,I}$ such that $pD'_{d,I}=\varphi_{n_t}(p)D'_{d,I}$.
		Then for $h\in D_{d,I}$ if 
		$$[h,p]\in \langle a_{-I_0+1},b_{-I_0+1},\cdots,a_{I_0},b_{I_0}\rangle\:\langle \{[x_i,y_j]\inv[x_{i+n_t}y_{j+n_t}]\mid x,y\in\{a,b\}\:i,j\in \ZZ\}\rangle\: C,$$
		then $[h,p]\in\langle \{[x_i,y_j]\inv[x_{i+n_t}y_{j+n_t}]\mid x,y\in\{a,b\}\:i,j\in \ZZ\}\rangle\: C $.
	\end{lemma}
	\begin{proof}
		We work in the group $\frac{G_{d,I}}{\langle \{[x_i,y_j]\inv[x_{i+n_t}y_{j+n_t}]\mid x,y\in\{a,b\}\:i,j\in \ZZ\}\rangle C}$. In this group, commutators of elements in $\{a_0,a_1,\cdots,b_0,b_1,\cdots\}$ are up to a sign determined by their type ($(a,a),(a,b)$ or $(b,b)$), by their diameter (the diameter of $[a_i,b_j]$ is given by $\abs{i-j}$), and by their congruence class ${\{i\mod n_t,j\mod n_t\}}$.
		
		Let $\tilde x_i=x_i x_{i+n_t} x_{i+2n_t}\cdots x_{i-n_t}$, then we can rewrite $[h,p]$ as $$
		\prod_{i=0}^{n_t-1} [h_{i,a},\tilde a_i][h_{i,b},\tilde b_i]
		$$
		for some $h_{i,x}\in D$.
		Notice that after applying relations of the form $[x_i,y_j]\inv[x_{i+n_t}y_{j+n_t}]$, we have that ${[a_i,\tilde b_j]=[\tilde a_i,b_j]}$, we may thus assume that $h_{i,b}$ lies in $\langle b_j\rangle$.
		Furthermore, again by applying the same relations, we may assume that ${h_{i,a}\in\langle a_0,b_0,\cdots,a_{n_t-1},b_{n_t-1}\rangle}$ and that ${h_{i,b}\in \langle b_0,b_1,\cdots,b_{n_t-1}\rangle.}$
		Let $l,k\in\{-I_0,\cdots,I_0\}$, we start with the $(a,a)$ case and show that the number of times $[a_l,a_k]$ appears in $$
		\prod_{i=0}^{n_t-1} [h_{i,a},\tilde a_i][h_{i,b},\tilde b_i]
		$$
		is equal to the number of times $[a_l,a_{k+2I_0}]$ appears. Indeed $[a_i,\tilde a_j]$ can contributes one of these factors precisely when either $i\cong l,j\cong k\mod n_t$ or $j\cong l,i\cong k\mod n_t$. In both of these cases, $[a_i,\tilde a_j]$ contributes to both of these factors in equal amount: once if $i\not\cong j\mod n_t$, and $0$ times otherwise (here we use that $k+2I_0-l<4I_0\leq\frac{I}{2}$ to guarantee that $[a_l,a_{k+2I_0}]$ and $[a_l,a_{2l-k-2I_0}]$ cancel out). From the above it follows that if $[a_l,a_{k+2I_0}]$ does not appear in $$
		\prod_{i=0}^{n_t-1} [h_{i,a},\tilde a_i][h_{i,b},\tilde b_i]
		$$ then neither does $[a_l,a_k]$.
		
		The $(b,b)$ case is the same. For the $(a,b)$ case, if $l=k$, then $[a_l,b_l]$ appears exactly half as often as $[a_l,b_{l+2I_0}]=[a_l,b_{l-2I_0}]$. In the other cases, they again appear equally often.
		
		We thus have that $$
		\prod_{i=0}^{n_t-1} [h_{i,a},\tilde a_i][h_{i,b},\tilde b_i]\in \langle a_{-I_0+1},b_{-I_0+1},\cdots,a_{I},b_{I}\rangle
		$$
		implies $$\prod_{i=0}^{n_t-1} [h_{i,a},\tilde a_i][h_{i,b},\tilde b_i]=1$$

	\end{proof}

	\begin{lemma}\label{prop:InD2CommutantLocal}
		Let $I=0$ or let $I>0$.
		Let $h\in \frac{D_{d,I}}{C}$ and $c\in \frac{D'_{d,I}}{C}$. Let $I_0$ be such that 
		$$h,c\in H_{I,I_0} \defeq \langle
		a_{-I_0+1}C,b_{-I_0+1}C,\cdots,a_{I_0}C,b_{I_0}C
		\rangle, $$
		Suppose there exists some $g\in \frac{D}{C}$ such that $[g,h]=c$, then there exists some $g'\in H_{I,I_0}$ such that $[g',h]=c$.
	\end{lemma}
	\begin{proof}
		Consider the map $\pi:\frac{D}{C}\rightarrow H_{I,I_0}$ such that $\pi(a_i)=a_i,\pi(b_i)=b_i,\forall i\in [-I_0+1,I_0]$ and such that $\pi(a_i)=\pi(b_i)=1$ otherwise. One checks that this map preserves the relations on $\frac{D}{C}$ and thus gives a well defined morphism. Furthermore, one checks that this map gives a retraction (with respect to the embedding of $H_{I,I_0}$ in $\frac{D}{C}$). We thus have $$c=\pi(c)=\pi([g,h])=[\pi(g),h].$$
	\end{proof}
	\begin{remark}\label{prop:CommutantLocallyIIndependent}
		In the above lemma, if $I,I'\geq2I_0$, then $H_{I,I_0}$ and $H_{I',I_0}$ are isomorphic by the unique isomorphism mapping $a_i$ to $a_i$ and $b_i$ to $b_i$ for all $i\in [-I_0+1,I_0]$.
	\end{remark}

		We proceed with the proof of \cref{prop:conjsep}

		Let $g_1,g_2\in G_d$ be non-conjugate. We will construct a finite quotient in which $g_1$ and $g_2$ remain non-conjugate. Notice that, as $\frac{G_d}{D}\cong\ZZ$ is conjugacy separable, we may assume that $g_1D$ and $g_2D$ are the same. We can thus find $n_t\in \ZZ$ and $h_1,h_2\in D$ such that $g_1=(h_1,n_t)$ and $g_2=(h_2,n_t)$. We demonstrate that there exists some $I$ such that $\pi_I(g_1)$ and $\pi_I(g_2)$ are non-conjugate. The result then follows as the groups $G_{d,I}$ are virtually nilpotent and thus conjugacy separable \cite{formanek1976conjugate}.
		
		We have a number of cases depending on which form $h_1,h_2$ and $n_t$ take.

		$\mathbf{1:\:g_1C\sim g_2C:}$\\
		
		After conjugating $g_2$, we may assume that $g_1C=g_2C$ or thus that $g_2g_1\inv\in C$. Furthermore as $g_1$ and $g_2$ are non-conjugate, we have that $g_2 g_1\inv\neq 1$.
		As $C$ is generated by $\{\cdots,c_{-2},c_{-1},c_1,c_2,\cdots\}$, there exists some $I$ such that $$g_2g_1\inv\in\langle c_{-I+1},\cdots,c_{-1},c_1,\cdots,c_I \rangle.$$
		This statement still holds after rounding up $I$ to some power of $2$.
		Consider then the quotient $\pi_0$ of $C$, obtained by the relations $c_i=c_{i+2I}$. If $2^{i}>I$, then $c_{2^{i}}$ gets identified with $c_0=1$. It follows that $\pi_0$ is injective on $$\langle c_{-I+1},\cdots,c_{-1},c_1,\cdots,c_I \rangle$$ and thus that $\pi_0(g_1\inv g_2)\neq 1$.
		
		Consider then $\pi_{2I}:G_d\rightarrow G_{d,2I}$. Suppose from contradiction that $\overline g_1 \sim \overline g_2$, then there exists some $g\in G_{d,2I}$ such that $g\overline g_1g\inv=\overline g_2$ or thus such that $[g,\overline g_1]=\overline g_2\overline g_1\inv $.
		As $g_2g_1\inv\in C$, \cref{prop:centAndCommutatorsDisjoint} implies that $[g,\overline g_1]$ must vanish. However $\pi_{2I}$ restricts to $\pi_0$ on $C$ and thus is $\overline g_2\overline g_1\inv$ non-trivial. By contradiction it must follow that $\overline g_1$ and $\overline g_2$ are non-conjugate.
		
		We are thus left with the case where $g_1C$ and $g_2C$ are non-conjugate.
		
		$\mathbf{2:\:g_1C\not\sim g_2C:}$\\
		 We distinguish further depending on the value of $n_t$
				
		$\mathbf{2.1:\:n_t\neq 0:}$\\
		Conjugating $g_2$ with itself clearly does not change anything, thus the conjugacy class of $g_2$ is given by $\bigcup_{i=0}^{\abs{n_t}-1}(g_2^{t^i})^D$. It thus suffices to find for each $i$ some quotient $\pi$ such that $\pi(g_1)\notin\pi(g_2^DC)$ where we replace $g_2$ with its conjugates $g_2^{t^i}$
		
		$\mathbf{2.1.1:\:g_1\notin g_2^DD':}$\\
		Let $I$ be some multiple of $n_t$ such that $$h_1D',h_2D'\in \langle a_{-I+1}D',b_{-I+1}D', \cdots,a_ID',b_ID' \rangle$$
		and let $K$ be the subgroup of $\frac{G_d}{D'}$ generated by $\{a_i\inv a_{i+2I},b_i\inv b_{i+2I}\mid i\in \ZZ\}$.
		Notice that $K$ is a subgroup of $\frac{[g_2,D]}{D'}$. Notice furthermore that $K$ is a normal subgroup of $\frac{G_d}{D'}$. We claim for $\pi:\frac{G_d}{D'}\rightarrow \frac{G_d}{D'K}$ that $\pi(g_1)\notin\pi(g_2^D)$. Indeed if $\pi(g_1)\in \pi(g_2^D)$, then there exist some $g\in d$ and some $k\in K$ such that $g_1kD'=g g_2 g\inv D'$ or thus such that $g_1g_2\inv kD'=[g,g_2]D'\in\frac{[g_2,D]}{D'}$. This in turn implies that there exists some $g'$ such that $g_1D'={g'} g_2{g'}\inv D'$. It follows that $\pi_I(g_1)\notin\pi_I(g_2^D)$. 
		
		$\mathbf{2.1.2:\:g_1\in g_2^DD':}$\\

		After conjugating with some element of $D$, we may assume that $g_1D'=g_2D'$.
		Let $I_0$ be some multiple of $n_t$ such that $$h_1,h_2\in \langle a_{-I_0+1},b_{-I_0+1}, \cdots,a_{I_0},b_{I_0} \rangle$$ and let $I=8I_0n_t$. We will show that $(h_1,n_t)^{D_I}$ and $(h_2,n_t)^{D_I}$ are distinct in $\frac{G_{d,I}}{C}$.
		Indeed suppose that $(p,0)$ is such that $(p,0)(h_1,n_t)(p,0)\inv C=(h_2,n_t)C$. Then at least ${\varphi_{n_t}(p)p\inv D'=h_1h_2\inv D'=D'.}$
		We thus have that $[p,h_1]p\varphi_{n_t}(p)\inv h_1C=h_2C$.
		Notice that $p\varphi_{n_t}(p)\in K_0\defeq\langle \{[x_i,y_j]\inv[x_{i+n_t}y_{j+n_t}]\mid i,j\in \ZZ,\:x,y\in\{a,b\}\}\rangle$. By \cref{prop:XPhiXVanishes}, we have $h_2h_1\inv\in K_0 C$, or thus when seen in $G_d$, we have that $g_2g_1\inv\in K_IK_0C$.
		Again using that $$g_2g_1\inv\in \langle a_{-I+1},b_{-I+1},\cdots,a_I,b_I \rangle,$$ we obtain that $g_2g_1\inv\in K_0C$ which thus implies that $g_2\in g_1^{D'}C$. From contraposition, the result follows.

		Each of the previous cases gives some $I_i$. For $I$, the least common multiple of these $I_i$, we find that $\pi_I(g_1)\nsim\pi_I(g_2)$.

		$\mathbf{2.2:\:n_t=0:}$\\
		Using the fact that $\frac{G_d}{D'}=\ZZ^2\wr\ZZ$ is conjugacy separable \cite{Remeslennikov1971Finite}, we may assume that $g_1D'=g_2D'$. 
		
		$\mathbf{2.2.1:\:g_1,g_2\notin D':}$\\
		
		For $g\in D\backslash D'$, let the support of $g$, $S(g)\subset \ZZ$ be the interval of minimal length such that ${gD'\in\langle \{a_iD',b_iD'\mid i\in S\}\rangle.}$
		Notice that $S(g)$ is non-empty and uniquely determined. Let $s(g)=\max(S(g))-\min(S(g))$ and suppose that $I>2s(g)+1$.
		Then the minimal interval $S_I(g)\subset \frac{\ZZ}{I\ZZ}$ such that $\pi_I(gD')\in\langle \{\pi_I(a_iD'),\pi_I(b_iD')\mid i\in S_I(g)\}\rangle$ is given by $S(g)\mod I$.
		Notice that for $h\in D$, we have $S(hgh\inv)=S(g)$. On the other hand for $j\neq 0$ we have $S(\varphi_{j}(g))\neq S(g)$ and for $j\neq I\ZZ$ that $S_I(\varphi_{j}(g))\neq S_I(g)$. But we still have that $s(g)=s(\varphi_j(g))$.
		
		Let now $I_0$ be minimal such that $$g_1,g_2g_1\inv\in H_{I_0} \defeq \langle 
		a_{-I_0+1},b_{-I_0+1},\cdots,a_{I_0},b_{I_0}
		\rangle, $$notice that $I_0\geq s(g_1)=s(g_2).$
		Let $I>2I_0+1$.
		Suppose now that ${\pi_I((h,n)(g_1,0)(h,n)\inv C)=\pi_I(g_2,0)C).}$ Then must ${S(\pi_I((h,n)(g_1,0)(h,n)\inv D'))=S(\pi_I((1,n)(g_1,0)(1,n)\inv D'))}$ be equal to $S((g_2,0)D')$. By the previous this only happens when $I\mid n$. As $(1,I)=(1,0)$ in $G_{d,I}$, we may thus assume that $n=0$. We thus have some $h\in D$ such that $\pi_I(hg_1h\inv C)=\pi_I(g_2C)$. By \cref{prop:InD2CommutantLocal}, we may assume that $h\in H_{I_0}$ and by \cref{prop:CommutantLocallyIIndependent} we thus have some $h\in G_{d}$ such that $hg_1h\inv C=g_2$.
		
		$\mathbf{2.2.2:\:g_1,g_2\in D':}$\\
		For $g\in D'\backslash C$ define $S(g)\subset \ZZ$ the minimal interval such that $gC\in \langle \{a_iC,b_iC\mid i\in S(g)\}\rangle$.
		Again $S(g)$ is non-empty. $S(g)$ is also uniquely determined: relators from $C$ may influence which commutators appear, but not the indices of the entries of those commutators. Again we define $s(g)=\max(S(g))-\min(S(g))$ and $S_I\subset \frac{\ZZ}{I\ZZ}$ as the minimal length interval such that ${\pi_I(gC)\in\langle \{\pi_I(a_iC),\pi_I(b_iC)\mid i\in S_I(g)\}\rangle.}$ Similar to before, if $I>2s(g)+1$ then $S_I(g)=S(g)\mod I$.
		Again conjugating with elements $h\in D'$ leaves $S(g)$ and $S_I(g)$ invariant, and conjugating with elements $(1,i)$ leaves $S$ invariant precisely when $i=0$ and $S_I$ invariant precisely when $I\mid i$.
		
		If $s(g_1)\neq s(g_2)$. Let $I>2\max(s(g_1),s(g_2))+1$, then $s_I(g_1)\neq s_I(g_2)$. As $s_I$ is invariant under conjugation in $\frac{G_{d,I}}{C}$, we have that $\pi_I(g_1C)$ and $\pi_I(g_2C)$ are non-conjugate.
		On the other hand is $s(g_1)=s(g_2)$, then after conjugating $g_1$ with some $(1,i)$ we may assume that $S(g_1)=S(g_2)$. For $I>2s(g_1)+1$, suppose now that $\pi_I((h,j)g_1(h,j)\inv)=\pi_I(g_2)$. As $S_I(g_1C)=S_I(g_2C)$ we have that $I\mid j$, and we may thus assume $j=0$. Furthermore as $D'$ centralises $D$, we thus have $\pi_I(g_1)=\pi_I(g_2)$. This is impossible as $\pi_I$ is injective on $\langle \{a_iC,b_iC\mid i\in S(g)\}\rangle$.
		
		Together, the above cases demonstrate \cref{prop:conjsep}
\section{The estimates}\label{sec:functions}

	Given the fact that $G_d$ is conjugacy separable, it is now not difficult to bound the conjugacy separability function from below.
	\begin{proposition}
		Suppose that $d$ is a prime valued, non-decreasing function, Let $S$ generate $G_d$, then for some constant $C>0$,
		$$\rf_{G_d,S}(n)\geq d(\lfloor\ln_2(Cn)\rfloor)$$
	\end{proposition}
	\begin{proof}
		As varying a generating set only changes the word norm up to a constant, the above statement is independent of a choice of generating set.
		Let $S=\{t,a_0,b_0\}$.
		Notice that $c_{2^{i}}$ is of word norm at most $8+2^{i+3}<2^{i+4}$.
		Furthermore suppose that $\varphi:G_d\rightarrow Q$ is some finite quotient such that $\varphi(c_{2^{i}})$ is non-trivial. As $c_{2^{i}}$ has prime order $d(i)$ it follows that $\varphi(c_{2^{i}})$ also must have order $d(i)$ and thus must $\abs{Q}\geq d(i)$. Substituting $n=2^i$ we thus have that $\rf_{G_d}(4n)\succ d(\log_2(n))$ from which the statement follows.
	\end{proof}
	As the conjugacy separability function is always larger than the residual finiteness growth, we immediately get the following:
	\begin{corollary}\label{prop:ConjSepEffective}
		Suppose that $d$ is a prime valued, non-decreasing function, Let $S$ generate $G_d$, then for some constant $C>0$,
		$$\conj_{G_d,S}(n)\geq d(\lfloor\ln_2(Cn)\rfloor)$$
	\end{corollary}
	\begin{definition}
		Let $d:\NN\rightarrow \NN$ be a non-decreasing function. We define conditions $\ref{cond:fastcompare}$ and $\ref{cond:fastcompute}$ on $d$ as follows:
		\begin{enumerate}[label=(\alph*)]
			\item It is possible to determine for $n,m\in\NN$ if $d(n)\geq m$ in time polynomial in $m$.\label{cond:fastcompare}
			\item It is possible to determine for $n\in\NN$ the value $d(n)$ in time polynomial in $d(n)$.\label{cond:fastcompute}
		\end{enumerate}
	\end{definition}
	\begin{lemma}\label{prop:LargeNiceExists}
		Let $d_0:\NN\rightarrow\NN$ be a computable function, then there exists a non-decreasing prime-valued function $d$ satisfying conditions \ref{cond:fastcompare} and \ref{cond:fastcompute} such that $d(n)\geq d_0(n)$ for all $n\in\NN$.
	\end{lemma}
	\begin{proof}
		We may assume that $d_0$ is non-decreasing.
		First we show there exists some non-decreasing (not necessarily prime-valued) function $d_1$ satisfying \ref{cond:fastcompare} and \ref{cond:fastcompute} such that $d_1(n)\geq d_0(n)$. Indeed let $\mc A$ be an algorithm computing $d_0$, let then $d_1(n)$ be given by $d_0(n)$ plus the number of steps needed to sequentially compute all the values $d_0(1), d_0(2),\cdots, d_0(n)$ using $\mc A$. Notice that the number of steps it takes to compute $d_0(i)$, can be computed in $2(d_0(i))$ steps by alternating a step from $\mc A$, and incrementing some counter. 
		One checks that $d_1$ is larger than $d_0$ and that it satisfies \ref{cond:fastcompute}. To demonstrate \ref{cond:fastcompare}, compute $d_0(i)$ as above, if the counter reaches $m$, then $d_0(i)$ must be at least $m$, and then this happens after $2m$ steps. otherwise the algorithm will halt in less than $2m$ steps.
		
		Let $d(n)$ now be the least prime number larger than $d_1(n)$.
		To compute $d(n)$, first compute $d_1(n)$, this can be done in at most $Cd_1(n)$ steps. Then for every integer $i$ larger than $d_1(n)$, check if $i$ is prime. Checking if $i$ is prime can be done naively in $Ci$ steps. The time to compute $d(n)$ is then up to constant factors given by $$
		d_1(n)+\sum_{i=d_1(n)}^{d(n)}i
		$$
		which can (again up to a factor) be bounded above by $d(n)^2$. This shows that \ref{cond:fastcompute} is satisfied.
		
		For \ref{cond:fastcompare}, if $m\leq d_1(n)$ then this can be checked in polynomial time as $d_1(n)$ itself satisfies \ref{cond:fastcompare}. If $d_1(n)\leq m\leq d(n)$, then this can be done by terminating the above computation after checking that all $i\in [d_1(n)+1,m]$ are not prime. Finally if a prime is found in $[d_1(n)+1,m]$, then this prime is found in no more than $Cd(n)^2\leq Cm^2$ steps.
	\end{proof}
	To solve the conjugacy problem on $G_d$, we first solve it on $\frac{G_d}{C}$:
	\begin{lemma}\label{prop:redconjprob}
		The group $\frac{G_d}{C}$ has conjugacy problem which is solvable in polynomial time. Furthermore if $g_1$ and $g_2$ are conjugate then one can find a conjugator $g$ such that $g\inv g_1g=g_2$ in polynomial time. This conjugator has word length at most polynomial in the lengths $\norm{g_1},\norm{g_2}$.
	\end{lemma}
	For the above it is useful to first construct ``fractional'' conjugators:
	The group $\frac{D_{d,0}}{C}$ is torsion free nilpotent, and thus has a Malcev completion $\QQ \frac{D}{C}$, that is a torsion-free nilpotent group $\QQ \frac{D}{C}$ in which $\frac{D}{C}$ embeds such that:
	\begin{itemize}
		\item for any $g\in \QQ \frac{D}{C}$, and any $n\in \ZZ_{\neq 0}$, there exists a unique $h\in \QQ \frac{D}{C}$ such that $h^n=g$\\
		\item for any $g\in \QQ \frac{D}{C}$ there exists some $n\in\ZZ_{\neq 0}$ such that $g^n\in \frac{D}{C}$.
	\end{itemize} 
	By \cite{segal1983polycyclic}, this group exists and is $2$-step nilpotent.

	\begin{proof}[Proof of \cref{prop:redconjprob}]
		Fix the generating set $S=\{t,a_0,b_0\}$ and suppose $\norm{g_1}_S,\norm{g_2}_S\leq n$.
		One checks in linear time whether $g_1D$ and $g_2D$ are equal. If they are not, then are $g_1$ and $g_2$ also non-conjugate, so we may assume $g_1D=g_2D\defeq n_t$. We distinguish $2$ cases.
		First suppose $n_t=0$. Let $h_1,h_2$ be such that $(h_i,n_t)=g_i$. Notice that $h_i$ can be expressed as a word in the letters $S_n\defeq\{a_i,b_i\mid \abs{i}\leq n\}$. This conversion can be done in linear time.
		
		One can in at most quadratic time determine the least $i_1$ such that $a_{i_1}$ or $b_{i_1}$ has non-trivial contribution in $h_1D'$. Similarly one determines $i_2$. 
		Notice that conjugating with elements of $D$ leaves $i_j$ invariant and conjugating with $t$ increments $i_j$ by one. In particular if ${(h,n)g_1(h,n)\inv =g_2}$ for some $(h,n)$, then must $n$ be equal to $i_2-i_1$. In linear time we may thus replace $g_1$ with $(1,i_2-i_1)g_1(1,i_2-i_1)\inv$ (just by relabelling the indices) this does not change the word norm of $g_1$ (with respect to $S_{3n}$). In the remainder we thus assume that $i_2=i_1$. We may check in at most quadratic time whether $h_1D'=h_2D'$. If this is not the case then we obtain that $g_1$ and $g_2$ are non-conjugate. Now to determine if some $h$ exists such that $hh_1h\inv=h_2$, one needs to find existence of a solution of a system of (up to a constant factor) $2n^2$ linear equations in $2n$ variables, where the entries take values at most $n^2$. This can be done in polynomial time in $n$.
		
		In the $n_t=0$ case, we are left with constructing a small conjugator. We first construct a conjugator in the Malcev completion $\QQ \frac{D}{C}$. Notice that for $q$ any rational number that if $h$ is a conjugator, then so is $hh_1^q$.
		
		Pick some $k$ and $l$ such that the contributions of $a_k$ and $b_l$ to $h_1D'$ are non-trivial and, if possible, such that $k=l$. By the previous observation, we may fix the $a_k$ coordinate of $h$ to a value of choice. Here we pick $a_k^0$. If $k=l$, then this uniquely determines the $b_k$ coordinate of $hD'$, namely $\frac{c_{[a_k,b_k]}}{c_{a_k}}$, where $c_{a_k}$ is the $a_k$ coordinate of $h_1$ and $c_{[a_k,b_k]}$ is the $[a_k,b_k]$ coordinate of $h_2h_1\inv$. On the other hand, suppose that $k\neq l$, then $c_{a_l}$ and $c_{b_k}$ must be trivial. In particular it follows that any $h$, the $[a_k,b_l]$ coordinate of $[h_1,h]$ is completely determined by the $a_k$ and $b_l$ coordinate of $h$. we obtain that the $b_l$ coordinate of $h$ is given by $\frac{c_{[a_k,b_l]}}{c_{a_k}}$. 
		
		Now we may in similar fashion to the first case, determine all $a_i$ coordinates of $h$ as $\frac{c_{[a_k,a_i]}}{c_{a_k}}$ and all $b_i$ coordinates of $h$ as $\frac{c_{[b_l,b_i]}c_{a_k}-c_{b_i}c_{[a_k,b_l]}}{c_{b_l}}$.

		If some conjugator exists, then we may multiply it with some power of $h_1$, and thus assume that its $a_k$ coefficient is an integer in the interval $[0,\abs{c_{a_k}}]$
		
		We can thus find this conjugator by iterating over all $i\in [0,c_{a_k}]$ and checking if $h'=h h_1^{\frac{i}{c_{a_k}}}D'$ belongs to $\frac{D}{C}$. Finding such an $i$ takes at most cubic time in $n$. One checks that all coordinates of $h'$ are at most cubic in $n$. The word norm $\norm{h'}_{S_n}$ is thus at most quartic in $n$ or thus is the word norm $\norm{(h',0)}_S$ at most quintic in $n$.
		
		Now assume $n_t\neq 0$, then $\abs{g_1D}$ is bounded above by $n$.

		Similar to the proof of \cref{prop:conjsep}, we may split in $n$ distinct cases in which we check whether $g_1^{t^{i}}\in g_2^{D}$ for $i\in[0,\abs{n_t}-1]$.
		For each $i$, finding $h\in \frac{D}{D'}$ such that $(h,0)g_1^{t^{i}}(h,0)\inv D'=g_2$ is again equivalent to solving a linear system. If such an $h$ exists then its $a_i$ coordinate must be given by $\sum_{j=0}^\infty c_{a_{i-jn_t}}$ where $c_{a_k}$ is the $a_k$ coordinate of $h_2h_1\inv$\footnote{Notice that this sum is finite as $h_2h_1\inv$ is finitely generated} . 
		Similarly the $b_i$ coordinate must be given by $\sum_{j=0}^\infty c_{b_{i-jn_t}}$. If the conjugator exists, then can $h$ only take non trivial values on $\{a_i,b_i\mid i\in [-n,n]\}$. It also follows that $h$ is of word norm at most cubic in $n$. By similar arguments, one finds (existence of) $h'\in D'$ such that $h'h(h_1,n_t)h\inv {h'}\inv=(h_2,n_t)$ where $h'$ is of length at most polynomial in $n$.
		These computations can all be carried out in polynomial time.

	\end{proof}
	\begin{proposition}\label{prop:conjprop}
		Let $d$ be a non-decreasing function satisfying \ref{cond:fastcompare} and \ref{cond:fastcompute}, then the conjugacy problem on $G_d$ is solvable in polynomial time.
	\end{proposition}
	\begin{proof}
		Let $g_1,g_2\in G_d$. Suppose that $g_1C$ and $g_2C$ are non-conjugate, then by \cref{prop:redconjprob}, this can be determined in at most polynomial time. Suppose $g_1C$ and $g_2C$ are conjugate, then we can find $g$ of length polynomial in $n$ such that $g\inv g_1gC=g_2C$. We claim now that $g_1\sim g_2$ if and only if $gg_1g\inv=g_2$. Indeed suppose that $g'$ is such that $g g_1 g\inv={g'} g_2{g'}\inv$. Then $g g_1 g\inv g_2\inv=[g',g_2]$. The former belongs to $C$ so by \cref{prop:centAndCommutatorsDisjoint}, it follows that $ g\inv g_1 gg_2\inv=[g',g_2]=e$.
		
		To determine conjugacy of $g_1$ and $g_2$, it suffices thus to determine if $g_2\inv g\inv g_1 g$ is trivial. This element can be expressed as $\sum_{i\in I}\gamma_ic_i$ where $I=[-n,n]$ and where the coefficients $\gamma_i$ are at most polynomial in $n$. For each $i\in I$ we may now check if $\gamma_i$ is trivial, or, when $i=2^j$ , if $\gamma_{i}$ is a multiple of $d(j)$. The last can be done by checking if $d(j)\leq \gamma_i$ and in that case determining $d(j)$ and checking if $d(j)$ divides $\gamma_i$ directly.
		If this is the case for all choices of $i$, then are $gg_1g\inv$ and $g_2$ equal and thus conjugate.
	\end{proof}
	\begin{lemma}\label{prop:checkIfQuotient}
		Let $d$ be a non-decreasing prime-valued function satisfying \ref{cond:fastcompare} and \ref{cond:fastcompute}.
		Given a finite group $Q$ with $3$ distinguished elements $\alpha,\beta,\tau$, then one checks whether there exists a group morphism $G_d\rightarrow Q$ mapping $a_0,b_0,t$ to $\alpha,\beta,\tau$ respectively. This can be done in time polynomial in the order of $Q$.
	\end{lemma}
	\begin{proof}
		Let $n_t$ be the order of $\tau$.
		For $i\leq n_t$, define $\alpha_i,\beta_i$ and $\gamma_i$ as the image of $a_i$, $b_i$ and $c_i$ respectively. These can be computed using at most $2i$, $2i$ and $8i+7$ multiplications each. 
		 One first needs to check whether $[[x,y],z]$ vanishes for all $x,y,z\in\{\alpha_i,\beta_j\mid i,j\in [1,n_t]\}$. Next one checks if the commutators $[\gamma_i,\tau]$ vanish. Finally one checks for all $c_{2^{i}}$ whether $\gamma_{2^{i}}^{d(i)}$ vanishes. For this, if $j$ is congruent to exactly $1$ $2$-power $2^{i}$, then one checks whether $\gamma_{2^{i}}^{d(i)}$ vanishes and if $j$ is congruent to at least $2$ such powers, then one checks whether $\gamma_{j}$ itself vanishes.
		For this one needs to check a number of things polynomial in $n_t$, which is bounded above by the order of $Q$.
	\end{proof}
	\begin{corollary}\label{prop:QuoteintEnumeration}
		If $d$ is non-decreasing prime-valued and satisfies \ref{cond:fastcompare} and \ref{cond:fastcompute}, then $G_d$ has quotient enumeration.
	\end{corollary}
	\begin{proof}
		Enumerate all finite groups, by \cref{prop:checkIfQuotient} we may check whether each of these groups can be a quotient of $G_d$.
	\end{proof}
	\begin{proof}[proof of \cref{prop:main}]
	Consider for the class of groups $G_d$ where $d:\NN\rightarrow\NN$ is a non-decreasing prime valued function satisfying \ref{cond:fastcompare} and \ref{cond:fastcompute}. These groups are conjugacy separable by \cref{prop:conjsep} and have quotient enumeration by \cref{prop:QuoteintEnumeration}. By \cref{prop:conjprop} they have conjugacy problem that is solvable in polynomial time and by combining  \cref{prop:ConjSepEffective} with \cref{prop:LargeNiceExists} we obtain that their conjugacy separability can be arbitrarily large.
	\end{proof}
	\bibliographystyle{plain}
	\bibliography{bibliography}
\end{document}